\documentclass[a4paper,fleqn]{cas-dc}

\usepackage[authoryear,longnamesfirst]{natbib}

\def\tsc#1{\csdef{#1}{\textsc{\lowercase{#1}}\xspace}}
\tsc{WGM}
\tsc{QE}

\usepackage{graphicx}          
\usepackage{amsmath}
\usepackage{amssymb}
\usepackage{amsfonts}
\usepackage{amsthm}
\usepackage{mathtools}
\usepackage{bm}
\usepackage{bbm}
\usepackage{xcolor}

\usepackage[normalem]{ulem}

\usepackage{enumerate}

\usepackage{subcaption}

\newtheorem{theorem}{Theorem}
\numberwithin{theorem}{section}
\newtheorem{lemma}[theorem]{Lemma}
\newtheorem{proposition}[theorem]{Proposition}

\newtheorem{assumption}[theorem]{Assumption}
\newtheorem{corollary}[theorem]{Corollary}
\newtheorem{example}[theorem]{Example}
\newtheorem{problem}[theorem]{Problem}

\newcommand{\mR}{\mathbb{R}}
\newcommand{\mS}{\mathbb{S}}

\newcommand{\st}{\text{s.t.}}

\DeclareMathOperator*{\tr}{tr}
\DeclareMathOperator*{\mE}{\mathbb{E}}
\DeclareMathOperator*{\mP}{\mathbb{P}}

\DeclareMathOperator*{\cov}{cov}

\newcommand{\rvu}{\bm{u}}

\newcommand{\Qtrue}{\bar{Q}}

\newcommand{\Rtrue}{\bar{R}}
\newcommand{\Ptrue}{\bar{P}}

\newcommand{\Ktrue}{\bar{K}}

\newcommand{\mfx}{\bm{x}}
\newcommand{\mfy}{\bm{y}}

\newcommand{\mfw}{\bm{w}}
\newcommand{\mfv}{\bm{v}}
\newcommand{\mfu}{\bm{u}}

\newcommand{\mfY}{\bm{Y}}

\newcommand{\g}{g}

\newcommand{\Ctrue}{\bar{C}}
\begin{document}
\let\WriteBookmarks\relax
\def\floatpagepagefraction{1}
\def\textpagefraction{.001}

\shorttitle{Inverse optimal control for averaged cost per stage linear quadratic regulators}    

\shortauthors{Han Zhang, Axel Ringh}  

\title [mode = title]{Inverse optimal control for averaged cost per stage linear quadratic regulators}  

\tnotemark[1] 

\tnotetext[1]{The work of Han Zhang is supported by National Natural Science Foundation (NNSF) of China under Grant 62103276,
and the work of Axel Ringh is supported by the Wallenberg AI, Autonomous Systems and Software Program (WASP) funded by the Knut and Alice Wallenberg Foundation.} 

%

\author[1]{Han Zhang}[orcid=0000-0002-3905-0633]

\cormark[1]


\ead{zhanghan_tc@sjtu.edu.cn}



\affiliation[1]{organization={Department of Automation, Shanghai Jiao Tong Univerisity},
            addressline={No. 800, Dongchuan Rd.}, 
            city={Shanghai},
            postcode={200240}, 
            country={China}}

\author[2]{Axel Ringh}[orcid=0000-0002-9778-1426]


\ead{axelri@chalmers.se}



\affiliation[2]{organization={Department of Mathematical Sciences, Chalmers University of Technology and the University of Gothenburg},
            city={Gothenburg},
            postcode={41296}, 
            country={Sweden}}

\cortext[1]{Corresponding author}



\begin{abstract}
Inverse Optimal Control (IOC) is a powerful framework for learning a behaviour from observations of experts.
The framework aims to identify the underlying cost function that the observed optimal trajectories (the experts' behaviour)  are optimal with respect to.
In this work, we considered the case of identifying the cost and the feedback law from observed trajectories generated by an ``average cost per stage" linear quadratic regulator.
We show that identifying the cost is in general an ill-posed problem, and give necessary and sufficient conditions for non-identifiability.
Moreover, despite the fact that the problem is in general ill-posed, 
we construct an estimator for the cost function and show that the control gain corresponding to this estimator is a statistically consistent estimator for the true underlying control gain.
In fact, the constructed estimator is based on convex optimization, and hence the proved statistical consistency is also observed in practice.
We illustrate the latter by applying the method on a simulation example from rehabilitation robotics.
\end{abstract}



\maketitle

\section{Introduction}\label{sec:introduction}
In nature, many decision processes can be modelled as optimal control problems \cite{alexander1996optima} since agents such as animals tend to behave optimally according to some criteria. 
This is closely connected to the concept of reinforcement learning \cite{kaelbling1996reinforcement, sutton2018reinforcement, bertsekas2019reinforcement}, and has found successful applications in many areas such as robotics and autonomous driving.
Notably, provided that its dynamics is a known a priori, the optimal behaviour of an agent is solely determined by its objective function. Hence, one of the fundamental problems in optimal control (and in reinforcement learning as well) is to design a good objective function (pay-off) so that the agent can achieve a good ``performance''.
The measure of ``performance'' should not solely be understood as a measure in terms of the objective function, but rather as an assessment of the generated behaviour in the actual environment.
The connection between the objective function design and ``performance'' is rather implicit, and it is usually hard to know which objective function gives a good ``performance''.
Therefore, the objective function designing process usually relies on ``trial-and-error'' and ``tuning".
One option to avoid this process is to model the objective function that an expert use when completing the same task, and then instead use the identified expert's objective function.
Moreover, once the expert's objective function is identified, we are also able to predict the expert's future behaviour.
First proposed in \cite{kalman1964linear}, such modelling process is known as Inverse Optimal Control (IOC). The topic has received plenty of interest in recent years, in particular under the name inverse reinforcement learning \cite{adams2022survey}.
Formally speaking, an IOC problem assumes that the observed data is governed by an optimal control model structure whose parameters in the objective function remains unknown.
Consequently, the IOC algorithm reconstructs the corresponding unknown parameters in the objective function based on observed data;
such problem can be categorized as gray-box system identification \cite[p.~13]{ljung1999system}.

As one of the classical optimal controller designs, the IOC problem for linear quadratic regulators has been studied extensively in different settings.
In particular, \cite{zhang2019inverse, zhang2019inverseCDC, zhang2021inverse, zhang2022indefinite} consider the discrete-time finite horizon case.
However, if the expert's task completion time is long, or if the observation is sampled ``too finely", it is computationally challenging to reconstruct the quadratic objective function. This is because the control gain is time-varying and generated by the Riccati iteration in the finite time-horizon case, and hence the number of variables to be optimized is proportional to the time-horizon length.
This motivates us to consider the IOC problem for ``averaged cost per stage" linear quadratic regulator, with both process noise and observation noise, which can be seen as an approximation of linear quadratic optimal control when the time-horizon length is long enough.

The IOC problem for infinite horizon linear quadratic regulators have been considered in both continuous-time \cite{boyd1994linear, jameson1973inverse, fujii1987new} and discrete-time \cite{priess2014solutions}. In particular, \cite{boyd1994linear, jameson1973inverse, fujii1987new, priess2014solutions} all give necessary and sufficient conditions for the correspondence between a given control gain $K$ and the objective function parameters.
In particular, both \cite{boyd1994linear} and \cite{priess2014solutions} formulate the necessary and sufficient conditions in terms of Linear Matrix Inequalities (LMIs) that involves a given control gain $K$.
However, this means that the 
IOC problem must be solved in two steps: first the control gain $K$ need to be identified, and then the necessary and sufficient condition (that involves $K$) can be used to estimate the corresponding parameters in the quadratic objective function; this is true also for the methods in \cite{jameson1973inverse, fujii1987new}. Notably, the prior knowledge that ``the observed data is generated by a linear quadratic regulator" (or formally stated, the model structure) has not been fully utilized during the identification process of $K$.  Hence, from a practical perspective, we are motivated to directly use the observed data to identify the underlying cost function.
Such methods have been considered in the infinite-horizon case, both in continuous-time \cite{xue2021inverse_inf} and discrete-time \cite{xue2021inverse}, respectively. However, in both cases the algorithm uses observations of both state and control signal. In some cases, the latter might not be available. Moreover, in both works the observations are noise-free, and there is no process noise in the dynamics.

Indeed, the IOC problem for linear quadratic regulators is a special case of more general nonlinear IOC problems \cite{molloy2018finite, molloy2020online, keshavarz2011imputing}. These methods are mainly based on minimizing the violation of Pontryagin maximum principle (Karush-Kuhn-Tucker conditions). However, as proved by \cite{aswani2018inverse}, such methods are not statistical consistent, which means that they are not robust to process noise and measurement noise. Since both process noise and measurement noise are inevitable in applications, we are motivated to develop an IOC algorithm for linear quadratic regulators that is robust against noise. The idea is to utilize the structure of linear quadratic regulators in order to obtain stronger results in terms of statistical consistency.

In this work, we consider the IOC problem of ``average  per stage cost" linear quadratic regulators. In particular, the system dynamics involves process noise and the state-trajectory observation is also contaminated by noise.
The contributions are three-fold:
\begin{enumerate}
\item We analyze the identifiability problem of the IOC problem and we show that the parameters in the objective function are in general not identifiable, even for $R = I$. Moreover, necessary and sufficient conditions for non-identifiability are given.
\item We construct an IOC algorithm for ``average  per stage cost" linear quadratic regulators. In particular, the algorithm is based on a convex optimization, and the size of the optimization problem is independent of the data amount. 
\item We prove the IOC algorithm recovers an estimate of the underlying cost together with a statistically consistent estimate of the control gain. In particular, this property is also attained in practice since the algorithm is based on convex optimization and hence do not suffer from issues due to local minima.
\end{enumerate}

The outline of the paper is as follows: in Section~\ref{sec:problem_formulation_and_identifiability} we formally introduce the forward problem and the inverse problem, and given necessary and sufficient conditions for non-identifiability of the cost function. Section~\ref{sec:IOC_alg} contains the main results; there, we derive the IOC algorithm, and show that the derived estimator for the control gain is a statistically consistent estimator. In Section~\ref{sec:num_sim}, the method is illustrated on a simulation example from rehabilitation robotics.
Finally, Section~\ref{sec:conclusion} contains some conclusions.

\paragraph*{Notations}
In general, boldface letters like $\mfx$ denotes random variables, while $x$ denotes a realization of the same random variable. $\mE_{\mfx}$ denotes mathematical expectation with respect to the random variable $\mfx$. We use $\overset{p}\rightarrow$ to denote convergence in probability and ``$a.s.$" is short for ``almost surely". 
Moreover, given a sequence of random variables $\bm{r}_n$, we use the notation $\bm{x}_n=o_p(\bm{r}_n)$ to denote $\bm{x}_n = \bm{y}_n \bm{r}_n$, where $\bm{y}_n\overset{p}\rightarrow 0$ (see, e.g., \cite{van1998asymptotic}).
Furthermore, $\mR^{n \times m}$ denotes the set of $n \times m$ matrices with real-valued entries, and $\mS^n$, $\mS^n_+$, and $\mS^n_{++}$ denote the subsets of symmetric matrices, of symmetric positive semi-definite, and of symmetric strictly positive definite matrices, respectively, all of size $n \times n$.
In addition, for $P_1,P_2\in\mS^n$, $P_1\succeq P_2$ and $P_1\succ P_2$ mean that $P_1-P_2\in\mS^n_+$ and $P_1-P_2\in\mS^n_{++}$, respectively.

\section{Problem formulation and identifiability discussions}\label{sec:problem_formulation_and_identifiability}
Let $(\Omega, \mathcal{F},\mP)$ be the probability space that carries random vectors $\bar{\mfx}\in\mR^n$, $\{\mfw_t\in\mR^n\}_{t=1}^\infty$, $\{\mfv_t\in\mR^n\}_{t=1}^\infty$,  $\{\mfy_t\in\mR^n\}_{t=1}^\infty$ whose distributions are unknown.
$\bar{\mfx}$ is the initial condition, $\mfw_t$ is process noise, $\mfv_t$ is observations noise, and $\mfy_t = \mfx_t+\mfv_t$ are the observations.
Given a realization $\bar{x}$ of the random vector $\bar{\mfx}$, suppose that the expert's decisions $\{\mfu_t\}_{t=1}^\infty$ over time are governed by an ``averaged cost per stage" linear quadratic regulator \cite[Sec.~4.6.5]{bertsekas2011dynamic}, i.e., by 
\begin{subequations}\label{eq:forward_problem}
\begin{align}
\min_{\substack{\{\mfx_t\}_{t=1}^\infty , \\ \{\mfu_{t}\}_{t=1}^\infty}}
  \;
  & \; \lim_{N\rightarrow\infty} \frac{1}{N}\mE_{\mfw_{1:N}}\Big\{\sum_{t=1}^N \frac{1}{2}\Big[\mfx_t^T\Qtrue\mfx_t+\mfu_t^T\Rtrue\mfu_t \Big]\Big\}\label{eq:forward_problem_cost}\\
  \st
  & \; \mfx_{t+1} = A\mfx_{t} + B\mfu_{t}+\mfw_t, \quad t = 1,2,\ldots, \label{eq:forward_problem_dynamics} \\
  & \; \mfx_1 = \bar{x}, \label{eq:forward_problem_init_cond}
\end{align}
\end{subequations}
where $A \in \mR^{n \times n}$, $B \in \mR^{n \times m}$, $\Qtrue\in\mathbb{S}^n_+$, $\Rtrue\in\mathbb{S}^m_{++}$. 
We refer to problem \eqref{eq:forward_problem} as ``the forward problem".
For the sake of simplicity, we assume $\Rtrue = I$ throughout the paper.
Before we proceed, the following mild assumptions are made regarding the ``forward problem" \eqref{eq:forward_problem}.

\begin{assumption}[Properties of the system dynamics]\label{ass:controlability_and_full_rank}
The system \eqref{eq:forward_problem_dynamics} is controllable and sampled from a continuous-time system. More specifically, this means that $A=e^{\hat{A}\Delta t}$ is invertible, where $\Delta t$ is the sampling period and $\hat{A}$ is the system matrix of the corresponding continuous-time system. Furthermore, the system \eqref{eq:forward_problem_dynamics} is not over-actuated, i.e., $B$ has full-column rank.
\end{assumption}

\begin{assumption}[I.I.D.~random variables]\label{ass:iid}
The random vectors $\{\mfw_t\}_{t=1}^\infty$ and $\{\mfv_t\}_{t=1}^\infty$ are white and mutually independent. In addition, both $\mfw_t$ and $\mfv_t$ are independent and identically distributed (I.I.D.) over time $t$, with $\mE[\mfw_t]=\mE[\mfv_t]=0$, $\forall t$. Moreover, it holds for their covariances $\Sigma_w:=\cov(\mfw_t,\mfw_t)$ and $\Sigma_v:=\cov(\mfv_t,\mfv_t),\:\forall t$ that $\|\Sigma_w\|_F<\infty$, $\|\Sigma_v\|_F<\infty$.
\end{assumption}

Another standard assumption in linear-quadratic optimal control is the following (cf.~\cite[Sec.~4.6.5]{bertsekas2011dynamic}).

\begin{assumption}[Observability]\label{ass:observability}
There exists a $k \leq n$ and a $\Ctrue \in \mR^{k \times n}$ such that $\Qtrue=\Ctrue^T\Ctrue$, and such that the pair $(A,\Ctrue)$ is observable.
\end{assumption}

Notably, Assumption~\ref{ass:observability} guarantees the existence of a unique stabilizing solution $\Ptrue\in\mS^n_+$ to the Discrete-time Algebraic Riccati Equation (DARE)
\cite[Prop.~4.4.1]{bertsekas2005dynamic}
\begin{align}
\Ptrue = A^T\Ptrue A+\Qtrue-A^T\Ptrue B(B^T\Ptrue B+I)^{-1}B^T\Ptrue A,\label{eq:DARE}
\end{align}
while the optimal control to \eqref{eq:forward_problem} is given by $\mfu_t = \Ktrue\mfx_t$, and where the control gain $\Ktrue$ is given by
\begin{align}
\Ktrue = -(B^T\Ptrue B+I)^{-1}B^T\Ptrue A.\label{eq:ctrl_gain}
\end{align}
In fact, under Assumptions~\ref{ass:controlability_and_full_rank} and \ref{ass:observability}, the following Lemma shows that unique stabilizing solution $\Ptrue\in\mS^n_+$ is strictly positive definite; this will be useful later on in the analysis.

\begin{lemma}\label{lem:P_pd}
Under Assumption~\ref{ass:controlability_and_full_rank} and \ref{ass:observability}, the DARE \eqref{eq:DARE} admits a unique stabilizing solution $\Ptrue$ that is \emph{strictly} positive definite.
\end{lemma}

\begin{proof}
This follows by an application of \cite[Thm.~2.1]{payne1973discrete}; details are omitted for brevity.
\end{proof}

Now, assumed that we, as an external observer, are able to observe the expert system's optimal trajectory.
However, such observations would inevitably be contaminated by some measurement noise; this is modeled as having observations of the form
\begin{align}
\mfy_t = \mfx_t+\mfv_t,\quad t=1,2,\ldots,\label{eq:observation}
\end{align}
where the observation noise $\mfv_t$ is as in Assumption~\ref{ass:iid}.
Suppose we observe $M$ independent trials $\{y_t^i\}_{t=1}^N$,  for $i=1, \ldots, M$, of the expert's system. More precisely, let $\mfy_t^i$, $i=1, \ldots, M$ be I.I.D.~and distributed as $\mfy_t$; what we observe are the realizations $y_t^i$ of the random vectors $\mfy_t^i$, $t=1, \ldots, N$ and $i=1, \ldots, M$.
Based on the observation \eqref{eq:observation} and aforementioned assumptions,
the goal of the IOC problem corresponding to the ``forward problem" \eqref{eq:forward_problem} is to identify the underlying $\Qtrue$ using the observed optimal trajectory $\{y_t^i\}_{t=1}^N,\:i=1,\ldots,M$.

Nevertheless, it turns out that the problem of identifying $\Qtrue$ using the observed optimal trajectory $\{y_t^i\}_{t=1}^N,\:i=1,\ldots,M$ is in general an ill-posed inverse problem, i.e., that the model structure is in general not identifiable.
This is in sharp contrast to the finite-time horizon case (see \cite{zhang2019inverse,zhang2021inverse, zhang2022indefinite}).
To show this, first note that if we regard the initial value $\bar{\mfx}$, the process noise $\mfw_t$, and the measurement noise $\mfv_t$ as inputs, and if we regard the observation $\mfy_t$ at each time step as output, then the model structure $\mathcal{M}(\Qtrue)$ that defines the input-output relation is determined by 
\begin{align}
\mfy_t = (A_{cl}(\Qtrue))^{t-1}\bar{\mfx}+\sum_{k=1}^{t-1}(A_{cl}(\Qtrue))^{t-1-k}\mfw_k+\mfv_t,
\label{eq:model_structure_def}
\end{align}
where the closed-loop system matrix $A_{cl}(\Qtrue):=A+B\Ktrue$ depends implicitly on $\Qtrue\in\mS^n_+$ via \eqref{eq:DARE} and \eqref{eq:ctrl_gain}. 
Hence the question of identifiability is: does there exists a $Q \neq \Qtrue$, where $Q \in \mS^n_+$, such that $A_{cl}(Q)=A_{cl}(\Qtrue)$? 
Before we prove necessary and sufficient conditions for when the question is answered in the affirmative, we present the following Lemma that will be useful in the analysis to come.

\begin{lemma}\label{lem:A_cl_invertible}
Under Assumption~\ref{ass:controlability_and_full_rank} and \ref{ass:observability}, the closed-loop system $A_{cl}(Q)$ is invertible for all $Q\in\mS^n_+$.
\end{lemma}

\begin{proof}
Using the Assumptions and \cite[Prop.~4.4.1]{bertsekas2005dynamic}, the DARE \eqref{eq:DARE} has a unique stabilizing solution and hence the closed-loop system matrix is well-defined.
The lemma can now be easily proved by following the same proof steps as the proof of \cite[Thm.~2.1]{zhang2019inverse},  and hence is omitted.
\end{proof}

The question of identifiability is answered in the following Proposition.

\begin{proposition}\label{prop:necessary_sufficient_cond_unidentifiable}
Under Assumption~\ref{ass:controlability_and_full_rank} and \ref{ass:observability}, the model structure $\mathcal{M}(\Qtrue)$ is not identifiable at $\Qtrue\in\mS^n_+$ if and only if there exists $\Delta P, \Delta Q\in\mS^n$ such that $\Delta P, \Delta Q\neq 0$ and such that the following holds:
\begin{subequations}\label{eq:suffcient_necessary_cond_unidentifiability}
\begin{align}
&B^T\Delta P = 0, \label{eq:B_transpose_delta_P_equals_zero}\\
&A^T\Delta PA-\Delta P+\Delta Q = 0, \label{eq:Qtrue_delta_P_minus_A_transpose_delta_P_A}\\
&\Ptrue+\Delta P\succeq 0,\label{eq:Ptrue_plus_delta_P}\\
&\Qtrue+\Delta Q\succeq 0,\label{eq:Qtrue_plus_delta_Q}
\end{align}
\end{subequations}
where $\Ptrue\in\mS^n_+$ is the unique stabilizing solution of DARE \eqref{eq:DARE} for $\Qtrue$.
\begin{proof}
We first prove the necessity. To this end, if $\mathcal{M}(\Qtrue)$ is not identifiable at $\Qtrue\in\mS^n_+$, it means that there exists a $Q\in\mS^n_+$ such that $Q\neq \Qtrue$, but for which the  DARE \eqref{eq:DARE} has a stabilizing solution $P\in\mS^n_+$ and such that $A_{cl}(Q)=A_{cl}(\Qtrue)$.
Denote $\Delta Q:=Q-\Qtrue$; this implies that $\Delta Q\neq 0$, that $\Delta Q\in\mS^n$, and that $\Qtrue+\Delta Q\succeq 0$, i.e., that \eqref{eq:Qtrue_plus_delta_Q} holds.
Similarly, denote $\Delta P := P-\Ptrue$, which means that $\Ptrue+\Delta P\succeq 0$, i.e., that \eqref{eq:Ptrue_plus_delta_P} holds.
Next, let $K(\Qtrue)$ and $K(Q)$ denote the control gains corresponding to $\Qtrue$ and $Q$, respectively, i.e., the control gains obtained by \eqref{eq:DARE} and \eqref{eq:ctrl_gain} with the corresponding cost matrix.
Under Assumption~\ref{ass:controlability_and_full_rank}, $B$ has full-column rank, from which it follows that 
\[
A+BK(Q) = A+BK(\Qtrue) \quad \Longleftrightarrow \quad K(Q) = K(\Qtrue).
\]
On the other hand, for any $\tilde{Q}\in\mS^n_+$ for which the  DARE \eqref{eq:DARE} has a stabilizing solution, for the corresponding solution $\tilde{P}$ and control gain $K(\tilde{Q})$  it holds that
\begin{align*}
& \; K(\tilde{Q}) = -(B^T \tilde{P} B+I)^{-1}B^T \tilde{P} A\\
\Rightarrow &\; (B^T \tilde{P}B+I)K(\tilde{Q})=-B^T\tilde{P}A\\
\Rightarrow &\; B^T\tilde{P}\underbrace{\left(A+BK(\tilde{Q})\right)}_{A_{cl}(\tilde{Q})} = -K(\tilde{Q}).
\end{align*}
Since by Lemma~\ref{lem:A_cl_invertible}, $A_{cl}(\tilde{Q})=A+BK(\tilde{Q})$ is invertible, it follows that $B^T \tilde{P} = -A_{cl}(\tilde{Q})^{-1}K(\tilde{Q})$ for all such $\tilde{Q}$. Therefore, since $A_{cl}(Q)=A_{cl}(\Qtrue)$ and $K(Q)=K(\Qtrue)$, for $P$ and $\Ptrue$ we have that
\begin{align}
B^TP = B^T\Ptrue\implies B^T\underbrace{(P-\Ptrue)}_{\Delta P} = 0,\label{eq:B_transpose_delta_P}
\end{align}
i.e., that \eqref{eq:B_transpose_delta_P_equals_zero} holds.
Moreover, since both $(\Ptrue,\Qtrue)$ and $(P=\Ptrue+\Delta P,Q=\Qtrue+\Delta Q)$ satisfy DARE \eqref{eq:DARE}, in view of \eqref{eq:B_transpose_delta_P}, we have that
\begin{align}
&\Ptrue+\Delta P = A^T(\Ptrue+\Delta P)A+A^T(\Ptrue+\Delta P)B\nonumber\\
&\left[B^T(\Ptrue+\Delta P)B+I\right]^{-1}B^T(\Ptrue+\Delta P)A +(\Qtrue+\Delta Q)\label{eq:perturbed_DARE}\\
&\implies \Delta P = A^T\Delta PA+\Delta Q\nonumber,
\end{align}
i.e., that \eqref{eq:Qtrue_delta_P_minus_A_transpose_delta_P_A} holds. 
This in turn implies that $\Delta P\neq 0$, otherwise by \eqref{eq:Qtrue_delta_P_minus_A_transpose_delta_P_A} it would contradict $\Delta Q\neq 0$.
This shows that all equations in \eqref{eq:suffcient_necessary_cond_unidentifiability} are satisfied by the non-zero difference matrices $\Delta P$ and $\Delta Q$. Hence the necessity is proved.

To prove the sufficiency, suppose that \eqref{eq:suffcient_necessary_cond_unidentifiability} holds. 
Since $(\Ptrue,\Qtrue)$ satisfies \eqref{eq:DARE}, and \eqref{eq:B_transpose_delta_P_equals_zero} and \eqref{eq:Qtrue_delta_P_minus_A_transpose_delta_P_A} hold, it is easy to see that \eqref{eq:perturbed_DARE} holds, which means that $(P=\Ptrue+\Delta P,Q=\Qtrue+\Delta Q)$ also satisfies DARE \eqref{eq:DARE}.
This means that we can generate an optimal control gain $K(Q=\Qtrue+\Delta Q)$ of \eqref{eq:forward_problem} with the underlying parameter $Q=\Qtrue+\Delta Q\in\mS^n_+$ by 
\begin{align*}
&K(Q) = -\left(B^T(\Ptrue+\Delta P)B+I\right)^{-1}B^T(\Ptrue+\Delta P)A\\
&=-(B^T\Ptrue B+I)^{-1}B^T\Ptrue A = K(\Qtrue), 
\end{align*}
where the second equality follows since \eqref{eq:B_transpose_delta_P_equals_zero} holds.
Since $\Qtrue$ and $Q$ generate the same control gain, they also generate the same closed-loop system. Moreover, since $\Delta Q \neq 0$, it must hold that $\Qtrue\neq Q$, and hence the model structure $\mathcal{M}(\bar{Q})$ is not identifiable at $\bar{Q}$.
\end{proof}
\end{proposition}

To illustrate the result of the Proposition, next we give an example where $\Qtrue=\Ctrue^T\Ctrue\in\mS^n_+$ is not identifiable, even when $(A,B,\Ctrue)$ is a minimal realization and $A$ is invertible.
\begin{example}\label{example:non_identifiable}
Consider
\[
A = \begin{bmatrix} 1 & 0 \\ 1 & 1 \end{bmatrix}, \quad
B = \begin{bmatrix} 1 \\ 0 \end{bmatrix}, \quad
\Ctrue = \begin{bmatrix} 0 & 1\end{bmatrix}.
\]
Notably, $A$ is invertible, $(A,B)$ is controllable and $(A, \Ctrue)$ is observable. Next, let 
\[
\Qtrue = \Ctrue^T\Ctrue = \begin{bmatrix}
    0 & 0 \\ 0 & 1
\end{bmatrix}.
\]
By Lemma~\ref{lem:P_pd} there exists a unique positive definite stabilizing solution $\Ptrue$ to the corresponding DARE \eqref{eq:DARE}.%
\footnote{A numerical computation in Matlab gives $\Ptrue \approx \! \begin{bmatrix} 3.3306 & \!\!\!\!\! 2.0810 \\ 2.0810 & \!\!\!\!\! 2.6005 \end{bmatrix} \succ 0$.}
We now construct $\Delta P, \Delta Q\in\mS^n$ such that $\Delta P, \Delta Q\neq 0$ and such that \eqref{eq:suffcient_necessary_cond_unidentifiability} holds.
To this end, let 
\[
\Delta P = \begin{bmatrix} \Delta p_1& \Delta p_2 \\ \Delta p_2 & \Delta p_3 \end{bmatrix}, \quad
\Delta Q = \begin{bmatrix} \Delta q_1 & \Delta q_2 \\ \Delta q_2 & \Delta q_3 \end{bmatrix}.
\]
From \eqref{eq:B_transpose_delta_P_equals_zero}, it follows that $\Delta p_1= \Delta p_2 = 0$. Plugging this into \eqref{eq:Qtrue_delta_P_minus_A_transpose_delta_P_A} gives 
\begin{align*}
\begin{bmatrix} 0 & 0 \\ 0 & 0 \end{bmatrix} & \!\!=\!\!
\begin{bmatrix} 1 & 1 \\ 0 & 1 \end{bmatrix}
\begin{bmatrix} 0 & 0 \\ 0 & \Delta p_3 \end{bmatrix}
\begin{bmatrix} 1 & 0 \\ 1 & 1 \end{bmatrix}
\!\!-\!\! \begin{bmatrix} 0 & 0 \\ 0 & \Delta p_3 \end{bmatrix}
\!\!+\!\! \begin{bmatrix} \Delta q_1 & \Delta q_2\\ \Delta q_2& \Delta q_3 \end{bmatrix} \\
& \!\!=\!\! \begin{bmatrix} \Delta p_3 & \Delta p_3 \\ \Delta p_3 & 0 \end{bmatrix}
+ \begin{bmatrix} \Delta q_1 & \Delta q_2 \\ \Delta q_2 & \Delta q_3 \end{bmatrix}
\end{align*}
and thus a $(\Delta P,\Delta Q)$ that satisfy \eqref{eq:B_transpose_delta_P_equals_zero} and \eqref{eq:Qtrue_delta_P_minus_A_transpose_delta_P_A} must be of the form
\begin{align}
\Delta P = \begin{bmatrix} 0 & 0 \\ 0 & -\alpha \end{bmatrix}, \quad
\Delta Q = \begin{bmatrix} \alpha & \alpha \\ \alpha & 0 \end{bmatrix}, \label{eq:delta_P_delta_Q_construction}
\end{align}
for some $\alpha \in \mR$.
Since under Assumption~\ref{ass:controlability_and_full_rank} and \ref{ass:observability}, $\Ptrue$ is strictly positive definite, for small enough values of $\alpha > 0$, $\Ptrue + \Delta P$ is still strictly positive definite. Moreover,
\[
\Qtrue + \Delta Q = \begin{bmatrix}
    \alpha & \alpha \\ \alpha & 1
\end{bmatrix},
\]
and for any $\alpha \in (0,1)$, we note that $[(\Qtrue + \Delta Q)]_{11} = \alpha > 0$ and that for the Schur complement of $\Qtrue + \Delta Q$ with respect to the same element we have that
\[
(\Qtrue + \Delta Q) \setminus [(\Qtrue + \Delta Q)]_{11} = 1 - \alpha \alpha^{-1} \alpha = 1 - \alpha > 0.
\]
This means that  $\Qtrue + \Delta Q$ is positive semi-definite \cite[Thm~1.12]{horn2005basic}.
In conclusion, this means that for $\alpha \in (0,1)$ small enough, the construction \eqref{eq:delta_P_delta_Q_construction} satisfy the necessary and sufficient condition \eqref{eq:suffcient_necessary_cond_unidentifiability} for non-identifiability. 
\end{example}

Equipped with the necessary and sufficient condition \eqref{eq:suffcient_necessary_cond_unidentifiability}, we have the following corollaries.

\begin{corollary}\label{cor:m_equals_n_identifiable}
Under Assumption~\ref{ass:controlability_and_full_rank} and \ref{ass:observability}, if $m = n$, then the model structure $\mathcal{M}(\Qtrue)$  given by \eqref{eq:model_structure_def} is always identifiable.
\end{corollary}

\begin{proof}
In this case, $B$ is a square matrix, and since it by Assumption~\ref{ass:controlability_and_full_rank} is full rank, it is invertible. The only $\Delta P$ that satisfies \eqref{eq:B_transpose_delta_P_equals_zero} is thus $\Delta P = 0$, and the result follows.
\end{proof}

\begin{corollary}\label{cor:strictly_positive_def_non_identifiable}
Under Assumption~\ref{ass:controlability_and_full_rank} and \ref{ass:observability}, if $n>m$, then the model structure $\mathcal{M}(\Qtrue)$ given by \eqref{eq:model_structure_def} is not identifiable at any $\Qtrue\in\mS^n_{++}$.
\end{corollary}

\begin{proof}
Since $\Qtrue\in\mS^n_{++}$, there exists a unique invertible $\Ctrue\in\mR^{n\times n}$ such that $\Qtrue = \Ctrue^T\Ctrue$ by Cholesky factorization.
Hence $(A,\Ctrue)$ is observable. By Lemma~\ref{lem:P_pd}, the unique stabilizing solution $\Ptrue$ of DARE \eqref{eq:DARE} that corresponds to $\Qtrue$ is \emph{strictly} positive definite.
On the other hand, since $n>m$, the dimension of $\ker(B^T)$ is strictly larger than zero and hence there exists a $v \neq 0$ such that $v\in\ker(B^T)$.
For such a $v$, let $\Delta P = \lambda vv^T \neq 0$, and let $\Delta Q = \lambda(vv^T-A^Tvv^TA)$ according to \eqref{eq:Qtrue_delta_P_minus_A_transpose_delta_P_A}. If $\Delta Q = 0$, then $A^Tv = v$. However, that would mean that $B^T(A^T)^tv = 0$ for all $t = 0, \ldots, n-1$, which is not possible since $(A, B)$ is controllable by Assumption~\ref{ass:controlability_and_full_rank}. Therefore, $\Delta Q \neq 0$.
Since both $\Qtrue$ and $\Ptrue$ are strictly positive definite, we can choose $\lambda$ small enough to let \eqref{eq:Ptrue_plus_delta_P} and \eqref{eq:Qtrue_plus_delta_Q} hold. With such a construction, the necessary and sufficient conditions for non-identifiability \eqref{eq:suffcient_necessary_cond_unidentifiability} are all satisfied, and the statement is proved.
\end{proof}

From the discussion above, we know that the model structure $\mathcal{M}(\Qtrue)$ is not identifiable at any strictly positive definite $\Qtrue$. What's more, from Example \ref{example:non_identifiable}, we see that the model structure also may be non-identifiable at positive semi-definite $\Qtrue$. 
Nevertheless, for the purpose ``predicting the agent’s behaviour", it is enough to provide a good estimate on the ``true" control gain $K(\Qtrue)$ and hence the closed-loop system matrix $A_{cl}(\Qtrue)$.
From this perspective, it is acceptable if we can identify a matrix $Q\in\mS^n_+$ that corresponds to the ``true" control gain $K(\Qtrue)$ even though $Q \neq \Qtrue$; in the case of non-identifiability, this is indeed the best that one can do.
Notably, the key issue is how to do the identification under the model structure $\mathcal{M}(\Qtrue)$, that is, using the prior knowledge that ``the observed data is generated by an average cost per stage linear quadratic regulator". By using this prior information, the IOC method should provide a more accurate and efficient estimate of the closed-loop system compared to directly identifying the closed-loop system matrix $A_{cl}(\Qtrue)$ without specifying the model structure.

To this end, we summarize the IOC problem that is considered in this paper as follows.
\begin{problem}\label{pro:IOC}
Suppose that the observed expert system is governed by ``averaged cost per stage" linear quadratic regulator \eqref{eq:forward_problem} and that Assumption~\ref{ass:controlability_and_full_rank}, \ref{ass:iid} and \ref{ass:observability} hold. Let the dynamics $(A,B)$ as well as the covariances $\Sigma_w$ and $\Sigma_v$ be a priori known.
Given the optimal state trajectory measurements $\{y_t\}$ that are realizations of the observations $\{\mfy_t\}$ as in \eqref{eq:observation}, estimate the control gain $K(\Qtrue)$ that corresponds to the $\Qtrue\in\mS^n_+$ in the objective function \eqref{eq:forward_problem_cost}.
\end{problem}

\section{IOC algorithm}\label{sec:IOC_alg}
In this section, we construct the IOC algorithm for the averaged cost per stage linear quadratic regulator. Moreover, we also prove that the constructed IOC algorithm is statistically consistent in the sense of the corresponding control gain converges. Namely, the control gain $K(Q_M^*)$ corresponding to the estimate $Q_M^*$ obtained using $M$ observed trajectories converge in probability to the ``true" control gain $K(\Qtrue)$ as the amount of data, $M$, tends to infinity. Furthermore, the constructed IOC algorithm is based on convex optimization, meaning that all optimizers are global ones. This alleviates the issue of local minima, and thus ensures that the theoretical statistical consistency is actually attained in practice.

\subsection{Construction and empirical approximation}
Since the corresponding optimal control gain $K(\Qtrue)$ is specified by \eqref{eq:ctrl_gain} and \eqref{eq:DARE}, it make sense to try minimize the ``violation" of the DARE in order to reconstruct the corresponding control gain. Here, we first explain this intuition behind the algorithm; rigorous analysis of the constructed algorithm is presented in the next subsection.

To this end, since the goal is to find $P,Q\in\mS^n_+$ that corresponds to the ``true" control gain $K(\Qtrue)$, they must satisfy the DARE \eqref{eq:DARE}.
Notably, the DARE \eqref{eq:DARE} is not linear with respect to $P$ and $Q$, and hence not convex as a constraint in an optimization problem. 
However, for $P, Q\in\mS^n_+$, $B^TPB+I$ is strictly positive definite and hence invertible. Further, suppose $P$ and $Q$ satisfy
\begin{align}
A^TP A\! -\! P\! +\!Q\! -\! A^TP B(B^TP B+I)^{-1}B^TP A\succeq 0, \label{eq:relaxed_DARE_inequality}
\end{align}
where we relax the DARE \eqref{eq:DARE} from an equality to an inequality in the cone $\mS^n_+$.
The goal is that the above inequality holds with equality, so that the ``violation" of the DARE \eqref{eq:DARE} is minimized, namely, zero.
In addition, note that $A^TP A-P+Q-A^TP B(B^TP B+I)^{-1}B^TP A$ is actually the Schur complement $H(Q,P)\backslash (B^TP B+I)$, where
\begin{align}\label{eq:def_H}
H (Q,P)= 
\begin{bmatrix}
B^TP B+I &B^TP A\\
A^TP B &A^TP A+Q-P
\end{bmatrix}.
\end{align}
$B^TPB+I$ is strictly positive definite since $P \in\mS^n_ +$,  so by \cite[Thm~1.12]{horn2005basic}, \eqref{eq:relaxed_DARE_inequality} holds if and only if $H(Q,P)\succeq 0$.
On the other hand, in view of \eqref{eq:ctrl_gain}, we can rewrite the inequality \eqref{eq:relaxed_DARE_inequality} as
\begin{align*}
(A\!+\!BK(Q))^T\Ptrue (A\!+\!BK(Q))\!+\!Q+K(Q)^TK(Q)\!-\!P\!\succeq\! 0
\end{align*}
If we pre- and post-multiply with the optimal state $\mfx_t$ at time step $t$, add $\tr(P\Sigma_w)$ and take expectation with respect to $\mfx_t$ on both hand sides of the above inequality, we have
\begin{align*}
&\mE_{\mfx_t}\Big[\mfx_t^T(A\!+\!BK(Q))^T P (A\!+\!BK(Q))\mfx_t\!+\!\mfx_t^TQ\mfx_t\\
&+\mfx_t^TK(Q)^TK(Q)\mfx_t\!-\!\mfx_t^TP\mfx_t\Big]\!+\tr(P\Sigma_w)\ge \tr(P\Sigma_w).
\end{align*}
Since, by Assumption~\ref{ass:iid}, $\mfw_t$ is I.I.D., zero-mean, and independent of $\mfx_t$,
using the fact that $\tr(P\Sigma_w) = \tr(P\mE_{\mfw_t}[\mfw_t\mfw_t^T]) = \mE_{\mfw_t}[\mfw_t^TP\mfw_t]$,
 we can further write the above inequality as
\begin{align*}
&\mE_{\mfx_t}\Big[\mE_{\mfw_t}\left[\left(\left(A\!+\!BK(Q)\right)\mfx_t\!+\!\mfw_t\right)^TP \left(\left(A\!+\!BK(Q)\right)\mfx_t\!+\!\mfw_t\right)\right]\!\!\\
&+\!\mfx_t^T\Qtrue\mfx_t\!+\mfx_t^TK(Q)^TK(Q)\mfx_t\!-\!\mfx_t^TP\mfx_t\Big]\ge \tr(P\Sigma_w),
\end{align*}
In addition, if $P,Q\in\mS^n_+$ satisfies DARE \eqref{eq:DARE} and $Q$ is such that $K(Q)=K(\Qtrue)$, in view of the fact $\mfu_t = K(\Qtrue)\mfx_t$, the above inequality is equivalent to
\begin{align*}
&\mE_{\mfx_{t+1}}[\mfx_{t+1}^TP\mfx_{t+1}]\! +\! \mE_{\mfx_t}[\mfx_t^TQ\mfx_t\!+\!\|\mfu_t\|^2-\mfx_t^TP\mfx_t]\! \ge\! \tr(P\Sigma_w)\\
&\Leftrightarrow\!\! \mE_{\mfx_{t+1}}[\mfx_{t+1}^TP\mfx_{t+1}]\! +\! \mE_{\mfx_t}[\mfx_t^TQ\mfx_t\!+\!\|\mfu_t\|^2\!-\!\mfx_t^TP\mfx_t]\! +\! \tr(Q\Sigma_v)\\
&\ge \tr(P\Sigma_w)+\tr(Q\Sigma_v)
\end{align*}
and the inequality in fact holds with equality.
Further, taking expectation with respect to $\mfv_{t:t+1}$ on both hand sides of the above inequality, in view of \eqref{eq:observation} and since, by Assumption~\ref{ass:iid}, $\mfv_t$ is I.I.D., zero-mean, and independent of $\mfx_t$, we can similarly write the above inequality as
\begin{align}
&\mE_{\mfv_{t:t+1}}[\mE_{\mfx_{t+1}}[\mfx_{t+1}^TP\mfx_{t+1}]\! +\! \mE_{\mfx_t}[\mfx_t^TQ\mfx_t\!+\!\|\mfu_t\|^2\!-\!\mfx_t^TP\mfx_t]\! +\! \tr(Q\Sigma_v)]\nonumber\\
&\ge \tr(P\Sigma_w)+\tr(Q\Sigma_v)\nonumber\\
&\Leftrightarrow \mE_{\mfv_{t+1}}[\mE_{\mfx_{t+1}}[(\mfx_{t+1}+\mfv_{t+1})^TP(\mfx_{t+1}+\mfv_{t+1})]]+\mE_{\mfx_t}[\|\mfu_t\|^2]\nonumber\\
&-\mE_{\mfv_t}[\mE_{\mfx_t}[(\mfx_t+\mfv_t)^TP(\mfx_t+\mfv_t)]]+\mE_{\mfv_t}[\mE_{\mfx_t}[(\mfx_t+\mfv_t)^TQ(\mfx_t+\mfv_t)]]\nonumber\\
&\ge \tr(P\Sigma_w)+\tr(Q\Sigma_v)\nonumber\\
&\Leftrightarrow \mE_{\mfy_{t:t+1}}[\mfy_{t+1}^TP\mfy_{t+1}-\mfy_t^T P\mfy_t+\mfy_t^TQ\mfy_t]+\mE_{\mfx_t}[\|\mfu_t\|^2]\nonumber\\
&\ge \tr(P\Sigma_w)+\tr(Q\Sigma_v),\label{eq:optimality_violation_at_t}
\end{align}
where we have use the fact that $\mE_{\mfv_{t+1}}[\mfv_{t+1}^TQ\mfv_{t+1}]=\mE_{\mfv_t}[\mfv_t^TQ\mfv_t]$ $=\tr(Q\mE_{\mfv_t}[\mfv_t\mfv_t^T])=\tr(Q\Sigma_v)$.
As a summary for the idea of the IOC algorithm construction, the inequality \eqref{eq:optimality_violation_at_t} holds for all $(Q,P)\in\mathscr{D}$ and all $t=1,2,\ldots$, where
\begin{align*}
\mathscr{D} = \{ & (Q,P) \in \mS_ +^n \times \mS_+^n \mid H(Q,P) \succeq 0, \text{ where $H(Q,P)$} \\
& \text{is given in } \eqref{eq:def_H}  \}.
\end{align*}
In particular, the inequality \eqref{eq:optimality_violation_at_t} should hold with equality if $(Q,P)$ satisfies the DARE \eqref{eq:DARE} and the data $\mfy_t$ is generated using the corresponding control gain $K(Q)$.
Hence we can define the ``violation" of the optimality at time step $t$ as a function $\psi_t:\mathscr{D}\mapsto \mR$, where
\begin{align}
\psi_t(Q,P) \!:=&\!\!\mE_{\mfy_{t:t+1}}\!\!\!\left[\mfy_{t+1}^TP\mfy_{t+1}\!-\!\mfy_t^T P\mfy_t\!+\!\mfy_t^TQ\mfy_t\right]\!-\!\tr(P\Sigma_w)\nonumber\\
&\!-\!\tr(Q\Sigma_v).\label{eq:psi_t}
\end{align}
Note that, compared to \eqref{eq:optimality_violation_at_t}, in \eqref{eq:psi_t} we have dropped $\mE_{\mfx_t}[\|\mfu_t\|^2]$ since it is just a constant.%
\footnote{Note that $\mfu_t$ should be interpreted in the same way as $\mfy_t$, namely as data, although we do not actually have access to it.} 
Intuitively, it now make sense to minimize the total ``violation" of the optimality for all time steps, namely,
\begin{equation}\label{eq:opt_prob}
\min_{Q,P}  \quad \Psi(Q,P) \quad \st  \quad (Q,P) \in \mathscr{D},
\end{equation}
where
\begin{align}
& \Psi(Q,P) := \sum_{t=1}^{N-1}\psi_t(Q,P)= \mE_{\mfy_{1:N}} \Big[ \mfy_N^T P \mfy_N  -  \mfy_1^T P \mfy_1  \nonumber\\
&\!\! +\!\! \sum_{t = 1}^{N-1} \mfy_t^T Q \mfy_t  \Big]\! -\! (N\!-\!1)\tr(Q \Sigma_{v})\!-\! (N\!-\!1)\tr(P \Sigma_{w}).\label{eq:Psi_obj_func}
\end{align}
However, the distribution of $\mfy_{1:N}$ is normally not known a priori in practice. Hence, we do not have the explicit expression for, and can hence not evaluate, the expectation in \eqref{eq:Psi_obj_func}. Consequently, we cannot solve \eqref{eq:opt_prob} directly. 
Nevertheless, given observations $y_t^i$, $i=1,\ldots, M$ that are realizations of random vectors $\mfy_t^i$, which are  I.I.D.~random vectors with the same distribution as $\mfy_t$, we can approximate the objective function \eqref{eq:Psi_obj_func} with the empirical approximation of the expectation. 
To this end, denote $\mfY_t = [\mfy_t^1,\cdots,\mfy_t^N]$. It is clear that
\begin{align*}
\sum_{i=1}^M \mfy_t^{iT}P\mfy_t^i \!=\!\! \sum_{i=1}^M\tr(P\mfy_t^i\mfy_t^{iT})\!=\!\tr(P\!\sum_{i=1}^M\mfy_t^i\mfy_t^{iT})\!=\!\tr(P\mfY_t\mfY_t^T),\\
\sum_{i=1}^M \mfy_t^{iT}Q\mfy_t^i \!=\!\! \sum_{i=1}^M\tr(Q\mfy_t^i\mfy_t^{iT})\!=\!\tr(Q\!\sum_{i=1}^M\mfy_t^i\mfy_t^{iT})\!=\!\tr(Q\mfY_t\mfY_t^T).
\end{align*}
Based on the above facts, we can empirically approximate the objective function $\Psi(\cdot)$ as
\begin{align}
&\Psi(\cdot)\approx \frac{1}{M}\sum_{i=1}^M \Big\{\mfy_N^{iT}P\mfy_N - \mfy_1^{iT}P\mfy_1^i +\sum_{t=1}^{N-1} \mfy_t^{iT}Q\mfy_t^i\Big\}\nonumber\\
&\qquad \! -\! (N\!-\!1)\tr(Q \Sigma_{v})\!-\! (N\!-\!1)\tr(P \Sigma_{w})\nonumber\\
&= \frac{1}{M}\tr( P \mfY_N \mfY_N^T)  - \frac{1}{M}\tr(P \mfY_1\mfY_1^T) \! +\! \frac{1}{M}\!\tr(Q \sum_{t = 1}^{N-1} \mfY_t\mfY_t^T)\nonumber\\
& \qquad \!-\! (N\!-\!1)\tr(Q \Sigma_{v}) \!-\! (N\!-\!1)\tr(P \Sigma_{w})\nonumber\\
&:=\Psi^{\mfY}_M(Q,P),\label{eq:Psi_empirical_approx}
\end{align}
However, this approximated objective function may not be bounded from below on $\mathscr{D}$ for some realizations $\{\mfY_t(\omega)\}_{t=1}^N = \{Y_t\}_{t=1}^N$. Therefore, we need to bound the feasible domain in order for the optimization problem to be well-posed. Thus, let $\varphi$ be some fixed (arbitrarily large) number,
and let 
\[
\mathscr{D}(\varphi) = \{ (Q,P) \in \mathscr{D} \mid \| Q \|_F \leq \varphi,\| P \|_F \leq \varphi  \}.
\]
Since any observed data is generated by a pair $(\Qtrue, \Ptrue)$ with finite norm, by selecting $\varphi$ sufficiently large we can guarantee that $(\Qtrue, \Ptrue) \in \mathscr{D}(\varphi)$.
Since it is expected that the empirical approximation \eqref{eq:Psi_empirical_approx} converge to the ``true" expectation \eqref{eq:Psi_obj_func}, and since $(\Qtrue,\Ptrue)$ is an optimizer to  \eqref{eq:Psi_obj_func} by construction, as long the domain bound $\varphi$ is sufficiently large, it should not (statistically) affect the control gain estimate; this will be made rigorous shortly, see Theorem~\ref{thm:statistical_consistency}.
As a summary, in practice the estimator reads
\begin{align}\label{eq:approx_opt_prob}
\min_{Q,P}  \quad \Psi^Y_M(Q,P) \quad \st & \quad (Q,P) \in \mathscr{D}(\varphi),
\end{align}
and we get an estimate $K(Q_M^\star,P_M^\star)$ by solving \eqref{eq:approx_opt_prob} with $\{\mfY_t\}_{t=1}^N$ evaluated at the observed realization $\{Y_t\}_{t=1}^N$, where $K(Q_M^\star,P_M^\star)$ is defined as \eqref{eq:ctrl_gain} and $(Q_M^\star,P_M^\star)$ is an optimizer of \eqref{eq:approx_opt_prob}. In particular, note that \eqref{eq:approx_opt_prob} is a convex optimization problem. Moreover, the matrices $\mfY_t \mfY_t^T,\forall t=1,\ldots,N$, are of dimension $n\times n$ and can be computed before solving the optimization problem \eqref{eq:approx_opt_prob}. Thus, the size of \eqref{eq:approx_opt_prob} does not grow with the increase of the number of observed trajectories $M$.

\subsection{Statistical analysis}
We now rigorously analyze the statistical properties of the constructed algorithm. In particular, we are interested in whether the constructed algorithm is truly statistically consistent in the sense of estimating the corresponding control gain $K(\Qtrue)$, and hence robust to the additive process noise $\mfw_t$ and measurement noise $\mfv_t$. 

Recall that in \eqref{eq:model_structure_def}, we see the initial value $\bar{\mfx}$ as the input of the model structure. To proceed, we first make the following Assumption regarding the initial value $\bar{\mfx}$ so that the model structure is persistently excited.
\begin{assumption}[Persistent excitation]\label{ass:persistent_excitation}
It holds that $\cov(\bar{\mfx},\bar{\mfx})\succ 0$ and $\mE[\|\bar{\mfx}\|^2]<\infty$.
\end{assumption}

Based on this, we would like to present the following Lemmas that is useful in the analysis to come. 
\begin{lemma}\label{lem:persistent_excitation}
Under Assumption \ref{ass:controlability_and_full_rank}, \ref{ass:observability} and \ref{ass:persistent_excitation}, it holds that $\mE_{\mfx_t}[\mfx_t\mfx_t^T]\succ 0$, and $\mE[\|\mfx_t\|^2]<\infty$, $\forall t>1$.
\end{lemma}
\begin{proof}
Recall that under Assumption \ref{ass:controlability_and_full_rank} and \ref{ass:observability}, by Lemma \ref{lem:A_cl_invertible}, the closed-loop system $A_{cl}(\Qtrue)$ is invertible.
Then following the proof idea of \cite[Lemma~4.3]{zhang2022indefinite}, the statement follows.
\end{proof}

Next, we show that any globally optimal solution to the convex problem \eqref{eq:opt_prob} corresponds to the true control gain used to generate the data.

\begin{lemma}\label{lem:bounded_from_below}
Suppose that $(\Qtrue,\Ptrue)$ is the ``true" underlying parameters of the ``averaged cost per stage" linear quadratic regulator \eqref{eq:forward_problem}, with corresponding optimal states $\mfx_{1:N}$ and trajectory observation $\mfy_{1:N}$. Under Assumption~\ref{ass:controlability_and_full_rank}, \ref{ass:iid}, \ref{ass:observability}, the objective function \eqref{eq:Psi_obj_func} is bounded from below by $-\mE_{\mfx_t}[\|\mfu_t\|^2]$ on its domain $\mathscr{D}$. Furthermore, let $(Q^\star,P^\star)$ be any optimal solution of the convex optimization problem \eqref{eq:opt_prob}. Under Assumption~\ref{ass:persistent_excitation}, it holds that $K(Q^\star,P^\star) = K(\Qtrue,\Ptrue)$, where $K(\cdot)$ is defined by \eqref{eq:ctrl_gain}.
\end{lemma}
\begin{proof}
By the construction in the previous section, we know that 
\begin{align*}
\Psi(\cdot) &= \sum_{t=1}^{N-1}\psi_t(\cdot) = \sum_{t=1}^{N-1}\mE_{\mfx_t} \Big[(A\mfx_t+B\mfu_t)^T P(A\mfx_t+B\mfu_t)\\
&\qquad-\mfx_t^TP\mfx_t+\mfx_t^TQ\mfx_t\Big]-\mE_{\mfx_t}[\|\mfu_t\|^2]\\
& = \sum_{t=1}^{N-1}
\begin{bmatrix}
\mfu_t^T &\mfx_t
\end{bmatrix} 
H(\cdot)
\begin{bmatrix}
\mfu_t \\ \mfx_t
\end{bmatrix}-\mE_{\mfx_t}[\|\mfu_t\|^2],
\end{align*}
where $H(\cdot)$ is defined in \eqref{eq:def_H}.
Since on the domain $\mathscr{D}$, $H(\cdot)$ is positive semi-definite and $B^TPB+I\succ 0$, it holds that $H(\cdot)\backslash (B^TPB+I)=A^TP A-P+Q-A^TP B(B^TP B+I)^{-1}B^TP A\succeq 0$ \cite[Thm~1.12]{horn2005basic}. Therefore, we can further write the objective function as
\begin{align*}
&\Psi(\cdot) = \sum_{t=1}^{N-1} \Big( -\mE_{\mfx_t}[\|\mfu_t\|^2 +\mE_{\mfx_t}\Big[
\begin{bmatrix}
\mfu_t^T &\mfx_t
\end{bmatrix} \\
&
\begin{bmatrix}
I \\
A^TPB(B^TPB+I)^{-1} &I
\end{bmatrix}
\!\!\!
\begin{bmatrix}
B^TPB+I \\ &H\backslash (B^TPB+I)
\end{bmatrix}\\
&
\begin{bmatrix}
I &(B^TPB+I)^{-1}B^TPA\\
 &I
\end{bmatrix}
\!\!
\begin{bmatrix}
\mfu_t \\ \mfx_t
\end{bmatrix}\Big] \Big)\\
&=\sum_{t=1}^{N-1} \Big(-\mE_{\mfx_t}[\|\mfu_t\|^2 + \mE_{\mfx_t}\left[\mfx_t^T \left(H(\cdot)\backslash (B^TPB+I)\right)\mfx_t\right] + \\
& \mE_{\mfx_t}\left[\|(B^TPB+I)^{-\frac{1}{2}}\left(\mfu_t+(B^TPB+I)^{-1}B^TPA\mfx_t\right)\|^2\right] \Big),
\end{align*}
and conclude that $\Psi(\cdot)\ge -\mE_{\mfx_t}[\|\mfu_t\|^2]$.
Moreover, the ``true" parameters $(\Qtrue, \Ptrue)$ attain the lower bound since $(\Qtrue, \Ptrue)$ satisfies the DARE \eqref{eq:DARE} and \eqref{eq:ctrl_gain}. 
This implies that for any optimal solution $(Q^\star,P^\star)$, the optimal value is $\Psi(Q^\star,P^\star)=-\mE_{\mfx_t}[\|\mfu_t\|^2$.
Since the second and the third term of the sum in the above equation are non-negative, it follows that 
\begin{subequations}
\begin{align}
&\mE_{\mfx_t}\Big[\|(B^TP^\star B\!+\!I)^{-\frac{1}{2}}\Big(\mfu_t\!+\!\underbrace{(B^TP^\star B\!+\!I)^{-1}B^TP^\star A}_{K(Q^\star,P^\star)}\mfx_t\Big)\|^2\Big] \nonumber\\
& = 0,\label{eq:optimal_control_obj}\\
&\mE_{\mfx_t}\left[\mfx_t^T \left(H(Q^\star,P^\star)\backslash (B^TP^\star B\!+\!I)\right)\mfx_t\right]\nonumber\\
& =\tr\! \left(H(Q^\star,P^\star)\backslash (B^TP^\star B\!+\!I)\right)\mE_{\mfx_t}[\mfx_t\mfx_t^T]) = 0,\; \forall t.\label{eq:riccati_error_obj}
\end{align}
\end{subequations}
By Lemma \ref{lem:persistent_excitation}, we know that $\mE_{\mfx_t}[\mfx_t\mfx_t^T]\succ 0$. Now, recall that $H(Q^\star,P^\star)\backslash (B^TP^\star B+I)\succ 0$. Therefore, by \eqref{eq:riccati_error_obj} it follows that $H(Q^\star,P^\star)\backslash (B^TP^\star B+I) = A^TP^\star A\! -\! P^\star\! +\!Q^\star\! -\! A^TP B(B^TP^\star B+I)^{-1}B^TP^\star A=0$, namely, $(Q^\star,P^\star)$ satisfies the DARE \eqref{eq:DARE}. This means that any optimal solution $(Q^\star,P^\star)$ to \eqref{eq:opt_prob} is a ``valid pair" for the ``averaged cost per stage" linear quadratic regulator \eqref{eq:forward_problem}.
On the other hand, by \eqref{eq:optimal_control_obj}, it follows that $\mfu_t+K(Q^\star,P^\star)\mfx_t=0$ holds a.s.. Since $\mfu_t = K(\Qtrue,\Ptrue)\mfx_t$, for $t=1$ it in particular follows that 
\begin{align*}
&K(Q^\star,P^\star) \bar{\mfx}= K(\Qtrue,\Ptrue) \bar{\mfx},\quad \text{a.s.}
\end{align*}
By post-multiplying with $ \bar{\mfx}^T$ and taking the expectation with respect to $\bar{\mfx}$, we get
\[
K(Q^\star,P^\star) \mE_{\bar{\mfx}}[\bar{\mfx}\bar{\mfx}^T] =  K(\Qtrue,\Ptrue)  \mE_{\bar{\mfx}}[\bar{\mfx}\bar{\mfx}^T].
\]
Using Lemma~\ref{lem:persistent_excitation}, it follows that $K(Q^\star,P^\star) = K(\Qtrue,\Ptrue)$, thus finishing the proof.
\end{proof}

Next, we show that the approximation \eqref{eq:Psi_empirical_approx} in fact converges in an appropriate sense.

\begin{lemma}[Uniform law of large numbers]\label{lem:uniform_law_of_large_numbers}
Under Assumption~\ref{ass:controlability_and_full_rank},\ref{ass:iid}, \ref{ass:observability} and \ref{ass:persistent_excitation}, for large enough $\varphi$, it holds for the empirically approximated estimator $\Psi^{\mfY}_M(\cdot)$ that $\sup_{(Q,P)\in\mathscr{D}(\varphi)} |\Psi_M^{\mfY}(\cdot)-\Psi(\cdot)|\overset{a.s.}\rightarrow 0$
as $M\rightarrow \infty$.
\end{lemma}
\begin{proof}
The statement can be proved by following a similar route of \cite[Proof of Lem.~4.2]{zhang2021inverse}. 
Denote $\Psi(\cdot) = \mE_{\mfy_{1:N}}[\psi(\cdot)]$.
Notably, $\mathscr{D}(\varphi)$ is compact, $\psi(\cdot)$ is linear with respect to either $(Q,P)$ or $\mfy_{1:N}$. Hence $\psi(\cdot)$ is continuous at each $(Q,P)\in\mathscr{D}(\varphi)$ for all $\mfy_{1:N}$ and measurable for $\mfy_{1:N}$ at each $(Q,P)$.
Similarly, we can show that $\psi(\cdot)$ is actually bounded above by some dominating function $d(\mfy_{1:N})$ and $\mE[d(\mfy_{1:N})]<\infty$. Thus every condition of \cite[Thm.~2]{jennrich1969asymptotic} is satisfied and the statement follows.
\end{proof}

Equipped with the Lemmas above, we are ready to present the main theorem of the work, namely, the statistical consistency of the control gain estimate. 
\begin{theorem}[Statistical consistency]\label{thm:statistical_consistency}
Under Assumption~\ref{ass:controlability_and_full_rank}, \ref{ass:iid}, \ref{ass:observability}, \ref{ass:persistent_excitation}, let $(Q_M^\star,P_M^\star)$ be an optimal solution to the estimator \eqref{eq:approx_opt_prob}. Moreover, let $K(\cdot)$ be defined by \eqref{eq:ctrl_gain}. For $\varphi$ large enough, it holds that $K(Q_M^\star,P_M^\star)\overset{p}\rightarrow K(\Qtrue,\Ptrue)$ as $M\rightarrow\infty$.
\end{theorem}
\begin{proof}
By Lemma~\ref{lem:bounded_from_below}, if $(Q^\star,P^\star)$ is an optimal solution to \eqref{eq:opt_prob}, then its corresponding control gain estimate $K(Q^\star,P^\star)$ defined by \eqref{eq:ctrl_gain} is the same as the ``true" control gain $K(\Qtrue,\Ptrue)$.
Here we will use the short-hand notation $\Ktrue$ for $K(\Qtrue,\Ptrue)$.
Hence we can define the set of all optimal solutions to \eqref{eq:opt_prob} as
\begin{align*}
\Xi(\Ktrue) :=  \{ & (Q,P) \in \mS_+^n \times \mS_+^n \mid K(Q,P) = \Ktrue \text{ is given by \eqref{eq:ctrl_gain},} \\
& \text{and where } (Q,P) \text{ satisfies \eqref{eq:DARE}}   \}.
\end{align*}
Moreover, note that for $\varphi$ large enough, $\Xi(\Ktrue) \cap \mathscr{D}(\varphi) \neq \emptyset$ since $(\Qtrue, \Ptrue) \in \Xi(\Ktrue) \cap \mathscr{D}(\varphi)$.
In addition, let $\Xi^{\mfY}_M$ be the set of optimal solutions to the estimator \eqref{eq:approx_opt_prob}.  

To prove the statistical consistency of the estimate $K(Q_M^\star,P_M^\star)$, we first prove that for all $(Q^\star_M,P^\star_M) \in \Xi^{\mfY}_M$ and all $(Q^\star,P^\star) \in \Xi(\Ktrue)$, it holds that
$\Psi(Q^\star_M,P^\star_M)  \overset{p}{\rightarrow}  \Psi(Q^\star,P^\star)$
as $M \to \infty$. 
To this end, for any $(Q^\star_M,P^\star_M)\in \Xi^{\mfY}_M$, which is optimal to \eqref{eq:approx_opt_prob}, it holds that 
\begin{equation}\label{eq:optimality_ineq_psi}
\Psi^{\mfY}_M(Q^\star_M,P^\star_M) \leq  \Psi_M^{\mfY}(\tilde{Q}^\star, \tilde{P}^\star)
\end{equation}
for all $(\tilde{Q}^\star, \tilde{P}^\star) \in \Xi(\Ktrue) \cap \mathscr{D}(\varphi)$. By  Lemma~\ref{lem:uniform_law_of_large_numbers} we have that
$\sup_{(Q,P) \in \mathscr{D}(\varphi)} |\Psi^{\mfY}_M(Q,P) - \Psi(Q,P)| \overset{a.s.}{\rightarrow} 0,$
which implies that
$\sup_{(Q,P) \in \mathscr{D}(\varphi)} |\Psi^{\mfY}_M(Q,P) - \Psi(Q,P)| \overset{p}{\rightarrow} 0.$
This in turn implies that for all $(\tilde{Q}^\star, \tilde{P}^\star) \in \Xi(\Ktrue) \cap \mathscr{D}(\varphi)$ we have that
\begin{equation}\label{eq:op1}
\Psi^{\mfY}_M(\tilde{Q}^\star, \tilde{P}^\star) = \Psi(\tilde{Q}^\star, \tilde{P}^\star) + o_P(1)
\end{equation}
uniformly, i.e., with the same sequence of random variables converging to 0 in probability as $M \to \infty$ for all  $(\tilde{Q}^\star, \tilde{P}^\star) \in \Xi(\Ktrue) \cap \mathscr{D}(\varphi)$.
By using \eqref{eq:op1} together with \eqref{eq:optimality_ineq_psi}, for all $(Q^\star_M,P^\star_M) \in \Xi^{\mfY}_M$,  all $(\tilde{Q}^\star, \tilde{P}^\star) \in \Xi(\Ktrue) \cap \mathscr{D}(\varphi)$, and all $(Q^\star,P^\star) \in \Xi(\Ktrue)$, it holds that
\begin{align}
& \Psi^Y_M(Q^\star_M,P^\star_M) \leq  \Psi_M^Y(\tilde{Q}^\star, \tilde{P}^\star) =  \Psi(\tilde{Q}^\star, \tilde{P}^\star) + o_P(1) \nonumber \\
& =  \Psi(Q^\star, P^\star) + o_P(1) \label{eq:psi_m_and_psi_ineq}
\end{align}
uniformly, where the last equality holds since $(\tilde{Q}^\star, \tilde{P}^\star)$,  $(Q^\star,P^\star) \in \Xi(\Ktrue)$. 
This in turn means that uniformly for all $(Q^\star_M,P^\star_M) \in \Xi^{\mfY}_M$ and all  $(Q^\star, P^\star) \in \Xi(\Ktrue)$,
\begin{align*}
0 & \leq \Psi(Q^\star_M,P^\star_M) - \Psi(Q^\star,P^\star) \\
& \leq \Psi(Q^\star_M,P^\star_M) - \Psi^Y_M(Q^\star_M,P^\star_M) + o_P(1) \\
& \leq | \Psi(Q^\star_M,P^\star_M) - \Psi^Y_M(Q^\star_M,P^\star_M) | + o_P(1) \\
& \leq \sup_{(Q,P) \in \mathscr{D}(\varphi)} |\Psi^Y_M(Q,P) - \Psi(Q,P)| + o_P(1).
\end{align*}
where the first inequality holds since $(Q^\star_M,P^\star_M) \in \mathscr{D}(\varphi) \subseteq \mathscr{D}$ and $(Q^\star,P^\star)$ is optimal to \eqref{eq:opt_prob}, and the second inequality comes from \eqref{eq:psi_m_and_psi_ineq}.
By Lemma~\ref{lem:uniform_law_of_large_numbers}, the last expression converges to $0$ in probability as $M \to \infty$, and hence for all $(Q^\star_M,P^\star_M) \in \Xi^{\mfY}_M$ and all $(Q^\star,P^\star) \in \Xi(\Ktrue)$, it holds that $\Psi(Q^\star_M,P^\star_M)  \overset{p}{\rightarrow}  \Psi(Q^\star,P^\star)$
as $M \to \infty$.

Next, for any $(Q, P) \in \mathscr{D}$, let $K(Q, P) := -(B^T P B+I)^{-1}B^T P A$, i.e., be compute analogously to \eqref{eq:ctrl_gain} although $(Q, P)$ might only satisfy \eqref{eq:relaxed_DARE_inequality} and not DARE \eqref{eq:DARE}.
Using the above result we now prove that for all $(Q^\star_M,P^\star_M) \in \Xi^{\mfY}_M$, $K(Q^\star_M,P^\star_M)  \overset{p}{\rightarrow} \Ktrue$ as $M \to \infty$, i.e., statistical consistency.
To this end,
by an adaptation of  \cite[Proofs of Lem.~3.1 and Prop.~4.1]{zhang2021inverse} we have that for all $(Q, P) \in \mathscr{D}$, 
\begin{align*}
& \Psi(Q, P) + \sum_{t=1}^{N-1}\mE_{\mfx_t} \left[\| \bar{\rvu}_t \|^2\right] \\
& \geq \!\! \sum_{t=1}^{N-1} \! \mE_{\mfx_t} \! \left[\mfx_t^{T} \! ( K(Q, P) \! - \! \bar{K})^T \! (B^TPB \! + \! I) (K(Q, P) \! - \! \bar{K})\mfx_t^{}\right]\nonumber\\
&= \!\! \sum_{t=1}^{N-1} \!\! \tr \! \left[ \! (K(Q, P) \! - \! \bar{K})^T \! (B^TPB \! + \! I) (K(Q, P) \! - \! \bar{K}) \! \mE_{\mfx_t} \! \left[ \! \mfx_t\mfx_t^T\right] \! \right]\ge 0.\label{eq:stochastic_H_lower_bound}
\end{align*}
Since $Q^\star_M,P^\star_M \in \mathscr{D}(\varphi) \subseteq \mathscr{D}$, this means that
\begin{align*}
& \Psi(Q^\star_M,P^\star_M)  + \sum_{t=1}^{N-1}\mE_{\mfx_t} \left[\| \bar{\rvu}_t \|^2\right] \\
& \geq \! \sum_{t=1}^{N-1} \!\! \tr \! \Big[ \! (K(Q^\star_M, P^\star_M) \! - \! \bar{K})^T \! (B^T P^\star_M B \! + \! I) (K(Q^\star_M, P^\star_M) \! - \! \bar{K}) \\
& \quad \mE_{\mfx_t} \! \left[ \! \mfx_t\mfx_t^T\right] \! \Big] \geq 0 = \Psi(Q^\star, P^\star) + \sum_{t=1}^{N-1}\mE\left[\| \bar{\rvu}_t \|^2\right].
\end{align*}
Since $\Psi(Q^\star_M,P^\star_M)  \overset{p}{\rightarrow}  \Psi(Q^\star,P^\star)$, we must have that
\begin{align*}
& \sum_{t=1}^{N-1} \! \tr \! \Big[ \! (K(Q^\star_M,P^\star_M) \! - \! \bar{K})^T \! (B^T P^\star_M B \! + \! I) (K(Q^\star_M,P^\star_M) \! - \! \bar{K}) \\
& \quad \mE_{\mfx_t} \! \left[ \! \mfx_t\mfx_t^T\right] \! \Big]  \overset{p}{\rightarrow} 0.
\end{align*}
Since $\mE_{\mfx_t} \! \left[ \! \mfx_t\mfx_t^T\right] \succ 0$ by Lemma~\ref{lem:persistent_excitation}, and since $B^T P^\star_M B + I \succeq I$, this means that
$K(Q^\star_M,P^\star_M)  \overset{p}{\rightarrow} \Ktrue$ as $M \to \infty$. 
\end{proof}

\section{Numerical simulation}\label{sec:num_sim}

In this section we illustrate the proposed IOC method through a numerical simulation. More precisely, we consider the following scenario in rehabilitation: the patient is seated at a desk in front of a screen. On the desk there is a crank, with viscosity (damping), which when rotated moves a pointer left-and-right on the screen. The patient is repeatedly asked rotate the crank to place the pointer in a specific position (see Fig.~\ref{fig:rehablitation_illustration}).
\begin{figure}[!htpb]
\centering
\includegraphics[width=0.3\textwidth]{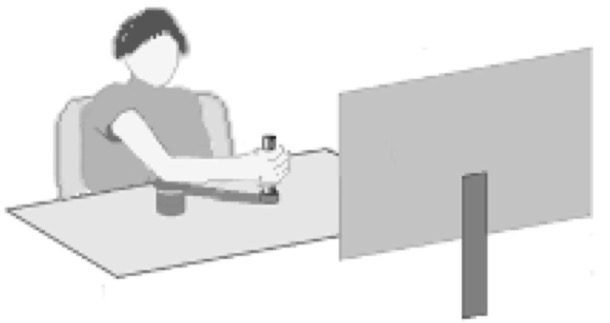}
\caption{The patient is asked to rotate a crank to place the pointer on the screen in a specific position.}
\label{fig:rehablitation_illustration}
\end{figure}

To model the dynamics of this set-up, let $\mfx_1,\mfx_2$ be the crank's angular position and the angular velocity, respectively, and let $\mfx = [\mfx_1,\mfx_2]^T$.  In addition, let $\mfu$ be the torque that is applied to the device by the patient.
A continuous-time physical model of such device is
\[
\dot{\mfx}(t) = \hat{A}\mfx(t)+\hat{B}\mfu(t)+\mfw(t),
\]
where $\mfw$ is
zero-mean (centered) white noise. In particular, we assume
\[
\hat{A} = \begin{bmatrix}
0 &1\\0 &-\alpha/\mathcal{I}
\end{bmatrix}=
\begin{bmatrix}
0 &1\\0 &-4
\end{bmatrix},
\quad
\hat{B} = \begin{bmatrix}
0\\1/\mathcal{I}
\end{bmatrix}=
\begin{bmatrix}
0 \\3
\end{bmatrix},
\]
where $\alpha$ is the viscosity (damping) coefficient and $\mathcal{I}$ is the moment of inertia.

To get a discrete-time system, the above system is discretized with a sampling period $\Delta t = 0.05$, namely, $A = e^{\hat{A}\Delta t}$, $B=\int_0^{\Delta t} e^{\hat{A}\tau}d\tau B$.
The procedure of the patient finishing such rehabilitation task is then modelled as ``average cost per stage" linear quadratic regulator \eqref{eq:forward_problem} with the discretized dynamics.  Notably in such formulation, the patient's target angular velocity is zero, and without loss of generality the target position is therefore set to zero, i.e., the patient is steering towards the origin in state space.
Specifically, for each time step, the additive noise process $\mfw_t$ and measurement noise $\mfv_t$ are drawn from multivariate normal distributions $\mathcal{N}(0,\Sigma_w)$ and $\mathcal{N}(0,\Sigma_v)$ with the covariances
\begin{align*}
 \Sigma_w \!\!=\!\! 10^{-4}
\begin{bmatrix}
0.1039    &0.0677 \\0.0677    &0.0997
\end{bmatrix}, \Sigma_v \!\!=\!\! 10^{-4}
\begin{bmatrix}
0.2328   &-0.2253\\-0.2253    &0.2180
\end{bmatrix}.
\end{align*}
These covariances were randomly generated in a similar way as those in \cite[Sec.~6]{zhang2022indefinite}.
Moreover, we observe the trajectories for $N=120$ time steps.

We generate 100 batches with 50000 trajectories in each, all with $\Qtrue = I$. The initial values $\bar{\mfx}$ are sampled from a uniform distribution on $[-\frac{2}{3}\pi,0]\times [-0.1,0.1]$.
In each batch, we divide the trajectories into a monotone increasing sequence of sets consisting of  $M=100+100(k-1)$ trajectories, for $k=1,\ldots,500$, i.e., the trajectories of the smaller sets are also contained in the larger sets. We let the domain bound be $\varphi=10^4$ and solve the inverse optimal control problem \eqref{eq:approx_opt_prob} for each set of trajectories in each batch. 
Hence, for each number of trajectories $M$ we estimate 100 control gains $K_{est}:=K(Q_M^\star,P_M^\star)$ (one for each batch) and can thus compute 100 relative errors $\frac{\|K_{est}-\Ktrue\|_F}{\|\Ktrue\|_F}$ of the estimation. From these 100 estimates, we empirically approximate the mean and standard deviation of the relative estimation error for each $M$.  The result is illustrated in a log-log plot in Fig.~\ref{fig:statistical_consistency}.

\begin{figure}[!htpb]
\centering
\includegraphics[trim=2.5cm 0cm 2.5cm 0cm, clip, width=\columnwidth]{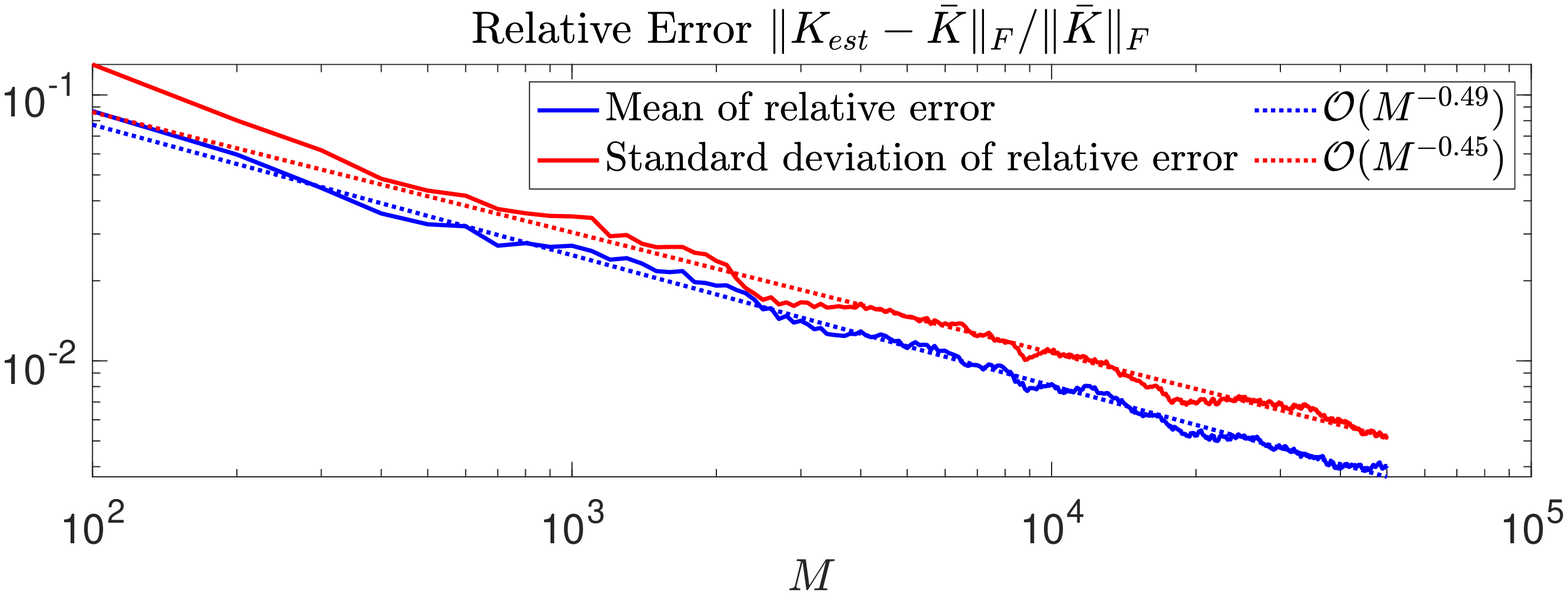}
\caption{Log-log plot of the mean and standard deviation of the relative error of $K_{est}$ v.s. the number of trajectories $M$.}
\label{fig:statistical_consistency}
\end{figure}

From Fig.~\ref{fig:statistical_consistency}, we can see that both the mean and the standard deviation of the relative estimation error decreases as $M$ increases; this is in accordance with the statement in Theorem~\ref{thm:statistical_consistency}. Furthermore, in the log-log plot, we see that both the mean and the standard deviation have an approximate linear decrease.
We thus fit the logarithmic data with an affine function and get a estimation of the convergence rate: $\texttt{Mean of relative error} \approx \mathcal{O}(M^{-0.49})$ and $\texttt{Standard deviation of relative error} \approx \mathcal{O}(M^{-0.45})$. 
The corresponding fitted affine functions are illustrated in Fig.~\ref{fig:statistical_consistency} with dashed lines.
Hence we suspect that the convergence rate of our proposed estimator is $\mathcal{O}(M^{-0.5})$, just like most M-estimators such as maximum log-likelihood \cite[p. 51]{van1998asymptotic}. The convergence rate analysis would be left for future work.

\section{Conclusions}\label{sec:conclusion}
In this work, we have considered the IOC problem for ``average cost per stage" linear quadratic regulators. The fact that the problem is in general ill-posed makes it impossible to recover the true underlying cost. Nevertheless, we constructed an IOC algorithm which recovers a statistically consistent estimate of the true underlying control gain.
There are many interesting continuations of this work: one would be to further investigate \eqref{eq:suffcient_necessary_cond_unidentifiability} to see if it is possible to derive more explicit conditions for when the problem is ill-posed. In particular, are all instances ill-posed, or are there problem instances that are not?
Another problem would be to try to prove that the convergence rate of the estimator is indeed $\mathcal{O}(M^{-0.5})$.
A third direction would be to formally investigate if the ill-posedness is an underlying factor for the  observed numerical ill-conditioning in the finite-time horizon case \cite[Sec.~5]{zhang2021inverse}. More specifically, for longer planning horizons $N$ in the finite-horizon setting, for $1 \approx t \ll N$ it would be expected that $P_t \approx \Ptrue$ and hence that $K_t \approx \Ktrue$. But that would mean that several different $\Qtrue + \Delta Q$ could give similar $K_t$. On the other hand, for $1 \ll t \approx N$, a possible issue is that the system is not as excited anymore, and hence the available data contains less information which thus leads to a numerical ill-conditioning at this end.

\printcredits

\bibliographystyle{cas-model2-names}

\bibliography{ref}

\begin{thebibliography}{27}
\expandafter\ifx\csname natexlab\endcsname\relax\def\natexlab#1{#1}\fi
\providecommand{\url}[1]{\texttt{#1}}
\providecommand{\href}[2]{#2}
\providecommand{\path}[1]{#1}
\providecommand{\DOIprefix}{doi:}
\providecommand{\ArXivprefix}{arXiv:}
\providecommand{\URLprefix}{URL: }
\providecommand{\Pubmedprefix}{pmid:}
\providecommand{\doi}[1]{\href{http://dx.doi.org/#1}{\path{#1}}}
\providecommand{\Pubmed}[1]{\href{pmid:#1}{\path{#1}}}
\providecommand{\bibinfo}[2]{#2}
\ifx\xfnm\relax \def\xfnm[#1]{\unskip,\space#1}\fi
\bibitem[{Adams et~al.(2022)Adams, Cody and Beling}]{adams2022survey}
\bibinfo{author}{Adams, S.}, \bibinfo{author}{Cody, T.},
  \bibinfo{author}{Beling, P.A.}, \bibinfo{year}{2022}.
\newblock \bibinfo{title}{A survey of inverse reinforcement learning}.
\newblock \bibinfo{journal}{Artificial Intelligence Review}
  \bibinfo{volume}{55}, \bibinfo{pages}{4307--4346}.
\bibitem[{Alexander(1996)}]{alexander1996optima}
\bibinfo{author}{Alexander, R.M.}, \bibinfo{year}{1996}.
\newblock \bibinfo{title}{Optima for animals}.
\newblock \bibinfo{publisher}{Princeton University Press},
  \bibinfo{address}{Princeton, NJ}.
\bibitem[{Aswani et~al.(2018)Aswani, Shen and Siddiq}]{aswani2018inverse}
\bibinfo{author}{Aswani, A.}, \bibinfo{author}{Shen, Z.J.},
  \bibinfo{author}{Siddiq, A.}, \bibinfo{year}{2018}.
\newblock \bibinfo{title}{Inverse optimization with noisy data}.
\newblock \bibinfo{journal}{Operations Research} \bibinfo{volume}{66},
  \bibinfo{pages}{870--892}.
\bibitem[{Bertsekas(2019)}]{bertsekas2019reinforcement}
\bibinfo{author}{Bertsekas, D.}, \bibinfo{year}{2019}.
\newblock \bibinfo{title}{Reinforcement learning and optimal control}.
\newblock \bibinfo{publisher}{Athena Scientific}, \bibinfo{address}{Nashua,
  NH}.
\bibitem[{Bertsekas(2005)}]{bertsekas2005dynamic}
\bibinfo{author}{Bertsekas, D.P.}, \bibinfo{year}{2005}.
\newblock \bibinfo{title}{Dynamic programming and optimal control: Volume 1}.
\newblock \bibinfo{edition}{3rd} ed., \bibinfo{publisher}{Athena scientific},
  \bibinfo{address}{Belmont, MA}.
\bibitem[{Bertsekas(2011)}]{bertsekas2011dynamic}
\bibinfo{author}{Bertsekas, D.P.}, \bibinfo{year}{2011}.
\newblock \bibinfo{title}{Dynamic programming and optimal control: Volume 2}.
\newblock \bibinfo{edition}{3rd} ed., \bibinfo{publisher}{Athena scientific},
  \bibinfo{address}{Belmont, MA}.
\bibitem[{Boyd et~al.(1994)Boyd, El~Ghaoui, Feron and
  Balakrishnan}]{boyd1994linear}
\bibinfo{author}{Boyd, S.}, \bibinfo{author}{El~Ghaoui, L.},
  \bibinfo{author}{Feron, E.}, \bibinfo{author}{Balakrishnan, V.},
  \bibinfo{year}{1994}.
\newblock \bibinfo{title}{Linear matrix inequalities in system and control
  theory}.
\newblock \bibinfo{publisher}{SIAM}, \bibinfo{address}{Philadelphia, PA}.
\bibitem[{Fujii(1987)}]{fujii1987new}
\bibinfo{author}{Fujii, T.}, \bibinfo{year}{1987}.
\newblock \bibinfo{title}{A new approach to the lq design from the viewpoint of
  the inverse regulator problem}.
\newblock \bibinfo{journal}{IEEE Transactions on Automatic Control}
  \bibinfo{volume}{32}, \bibinfo{pages}{995--1004}.
\bibitem[{Horn and Zhang(2005)}]{horn2005basic}
\bibinfo{author}{Horn, R.A.}, \bibinfo{author}{Zhang, F.},
  \bibinfo{year}{2005}.
\newblock \bibinfo{title}{Basic properties of the {S}chur complement}, in:
  \bibinfo{editor}{Zhang, F.} (Ed.), \bibinfo{booktitle}{The {S}chur Complement
  and Its Applications}. \bibinfo{publisher}{Springer},
  \bibinfo{address}{Boston, MA}, pp. \bibinfo{pages}{17--46}.
\bibitem[{Jameson and Kreindler(1973)}]{jameson1973inverse}
\bibinfo{author}{Jameson, A.}, \bibinfo{author}{Kreindler, E.},
  \bibinfo{year}{1973}.
\newblock \bibinfo{title}{Inverse problem of linear optimal control}.
\newblock \bibinfo{journal}{SIAM Journal on Control} \bibinfo{volume}{11},
  \bibinfo{pages}{1--19}.
\bibitem[{Jennrich(1969)}]{jennrich1969asymptotic}
\bibinfo{author}{Jennrich, R.I.}, \bibinfo{year}{1969}.
\newblock \bibinfo{title}{Asymptotic properties of non-linear least squares
  estimators}.
\newblock \bibinfo{journal}{The Annals of Mathematical Statistics}
  \bibinfo{volume}{40}, \bibinfo{pages}{633--643}.
\bibitem[{Kaelbling et~al.(1996)Kaelbling, Littman and
  Moore}]{kaelbling1996reinforcement}
\bibinfo{author}{Kaelbling, L.P.}, \bibinfo{author}{Littman, M.L.},
  \bibinfo{author}{Moore, A.W.}, \bibinfo{year}{1996}.
\newblock \bibinfo{title}{Reinforcement learning: A survey}.
\newblock \bibinfo{journal}{Journal of artificial intelligence research}
  \bibinfo{volume}{4}, \bibinfo{pages}{237--285}.
\bibitem[{Kalman(1964)}]{kalman1964linear}
\bibinfo{author}{Kalman, R.E.}, \bibinfo{year}{1964}.
\newblock \bibinfo{title}{When is a linear control system optimal?}
\newblock \bibinfo{journal}{Journal of Basic Engineering} \bibinfo{volume}{86},
  \bibinfo{pages}{51--60}.
\bibitem[{Keshavarz et~al.(2011)Keshavarz, Wang and
  Boyd}]{keshavarz2011imputing}
\bibinfo{author}{Keshavarz, A.}, \bibinfo{author}{Wang, Y.},
  \bibinfo{author}{Boyd, S.}, \bibinfo{year}{2011}.
\newblock \bibinfo{title}{Imputing a convex objective function}, in:
  \bibinfo{booktitle}{2011 IEEE international symposium on intelligent
  control}, \bibinfo{organization}{IEEE}. pp. \bibinfo{pages}{613--619}.
\bibitem[{Ljung(1999)}]{ljung1999system}
\bibinfo{author}{Ljung, L.}, \bibinfo{year}{1999}.
\newblock \bibinfo{title}{System Identification (2nd Ed.): Theory for the
  User}.
\newblock \bibinfo{publisher}{Prentice Hall PTR}, \bibinfo{address}{USA}.
\bibitem[{Molloy et~al.(2018)Molloy, Ford and Perez}]{molloy2018finite}
\bibinfo{author}{Molloy, T.L.}, \bibinfo{author}{Ford, J.J.},
  \bibinfo{author}{Perez, T.}, \bibinfo{year}{2018}.
\newblock \bibinfo{title}{Finite-horizon inverse optimal control for
  discrete-time nonlinear systems}.
\newblock \bibinfo{journal}{Automatica} \bibinfo{volume}{87},
  \bibinfo{pages}{442--446}.
\bibitem[{Molloy et~al.(2020)Molloy, Ford and Perez}]{molloy2020online}
\bibinfo{author}{Molloy, T.L.}, \bibinfo{author}{Ford, J.J.},
  \bibinfo{author}{Perez, T.}, \bibinfo{year}{2020}.
\newblock \bibinfo{title}{Online inverse optimal control for
  control-constrained discrete-time systems on finite and infinite horizons}.
\newblock \bibinfo{journal}{Automatica} \bibinfo{volume}{120},
  \bibinfo{pages}{109109}.
\bibitem[{Payne and Silverman(1973)}]{payne1973discrete}
\bibinfo{author}{Payne, H.}, \bibinfo{author}{Silverman, L.},
  \bibinfo{year}{1973}.
\newblock \bibinfo{title}{On the discrete time algebraic riccati equation}.
\newblock \bibinfo{journal}{IEEE Transactions on Automatic Control}
  \bibinfo{volume}{18}, \bibinfo{pages}{226--234}.
\bibitem[{Priess et~al.(2014)Priess, Conway, Choi, Popovich and
  Radcliffe}]{priess2014solutions}
\bibinfo{author}{Priess, M.C.}, \bibinfo{author}{Conway, R.},
  \bibinfo{author}{Choi, J.}, \bibinfo{author}{Popovich, J.M.},
  \bibinfo{author}{Radcliffe, C.}, \bibinfo{year}{2014}.
\newblock \bibinfo{title}{Solutions to the inverse {LQR} problem with
  application to biological systems analysis}.
\newblock \bibinfo{journal}{IEEE Transactions on control systems technology}
  \bibinfo{volume}{23}, \bibinfo{pages}{770--777}.
\bibitem[{Sutton and Barto(2018)}]{sutton2018reinforcement}
\bibinfo{author}{Sutton, R.S.}, \bibinfo{author}{Barto, A.G.},
  \bibinfo{year}{2018}.
\newblock \bibinfo{title}{Reinforcement learning: An introduction}.
\newblock \bibinfo{edition}{2nd} ed., \bibinfo{publisher}{MIT Press},
  \bibinfo{address}{Cambridge, MA}.
\bibitem[{van~der Vaart(1998)}]{van1998asymptotic}
\bibinfo{author}{van~der Vaart, A.W.}, \bibinfo{year}{1998}.
\newblock \bibinfo{title}{Asymptotic statistics}.
\newblock \bibinfo{publisher}{Cambridge university press},
  \bibinfo{address}{Cambridge, United Kingdom}.
\bibitem[{Xue et~al.(2021a)Xue, Kolaric, Fan, Lian, Chai and
  Lewis}]{xue2021inverse_inf}
\bibinfo{author}{Xue, W.}, \bibinfo{author}{Kolaric, P.}, \bibinfo{author}{Fan,
  J.}, \bibinfo{author}{Lian, B.}, \bibinfo{author}{Chai, T.},
  \bibinfo{author}{Lewis, F.L.}, \bibinfo{year}{2021}a.
\newblock \bibinfo{title}{Inverse reinforcement learning in tracking control
  based on inverse optimal control}.
\newblock \bibinfo{journal}{IEEE Transactions on Cybernetics}
  \bibinfo{volume}{52}, \bibinfo{pages}{10570--10581}.
\bibitem[{Xue et~al.(2021b)Xue, Lian, Fan, Kolaric, Chai and
  Lewis}]{xue2021inverse}
\bibinfo{author}{Xue, W.}, \bibinfo{author}{Lian, B.}, \bibinfo{author}{Fan,
  J.}, \bibinfo{author}{Kolaric, P.}, \bibinfo{author}{Chai, T.},
  \bibinfo{author}{Lewis, F.L.}, \bibinfo{year}{2021}b.
\newblock \bibinfo{title}{Inverse reinforcement q-learning through expert
  imitation for discrete-time systems}.
\newblock \bibinfo{journal}{IEEE Transactions on Neural Networks and Learning
  Systems} .
\bibitem[{Zhang et~al.(2019a)Zhang, Li and Hu}]{zhang2019inverseCDC}
\bibinfo{author}{Zhang, H.}, \bibinfo{author}{Li, Y.}, \bibinfo{author}{Hu,
  X.}, \bibinfo{year}{2019}a.
\newblock \bibinfo{title}{Inverse optimal control for finite-horizon
  discrete-time linear quadratic regulator under noisy output}, in:
  \bibinfo{booktitle}{2019 IEEE 58th Conference on Decision and Control (CDC)},
  \bibinfo{organization}{IEEE}. pp. \bibinfo{pages}{6663--6668}.
\bibitem[{Zhang and Ringh(2022)}]{zhang2022indefinite}
\bibinfo{author}{Zhang, H.}, \bibinfo{author}{Ringh, A.}, \bibinfo{year}{2022}.
\newblock \bibinfo{title}{Statistically consistent inverse optimal control for
  discrete-time indefinite linear-quadratic systems}.
\newblock \bibinfo{journal}{arXiv preprint arXiv:2212.08426} .
\bibitem[{Zhang and Ringh(2023)}]{zhang2021inverse}
\bibinfo{author}{Zhang, H.}, \bibinfo{author}{Ringh, A.}, \bibinfo{year}{2023}.
\newblock \bibinfo{title}{Inverse linear-quadratic discrete-time finite-horizon
  optimal control for indistinguishable homogeneous agents: A convex
  optimization approach}.
\newblock \bibinfo{journal}{Automatica} \bibinfo{volume}{148},
  \bibinfo{pages}{110758}.
\bibitem[{Zhang et~al.(2019b)Zhang, Umenberger and Hu}]{zhang2019inverse}
\bibinfo{author}{Zhang, H.}, \bibinfo{author}{Umenberger, J.},
  \bibinfo{author}{Hu, X.}, \bibinfo{year}{2019}b.
\newblock \bibinfo{title}{Inverse optimal control for discrete-time
  finite-horizon linear quadratic regulators}.
\newblock \bibinfo{journal}{Automatica} \bibinfo{volume}{110},
  \bibinfo{pages}{108593}.

\end{thebibliography}

\end{document}